\documentclass[hidelinks]{article}
\usepackage{graphicx} 
\usepackage[a4paper, margin=2.5cm]{geometry}
\usepackage{hyperref}
\usepackage{amsthm}
\usepackage{thm-restate}
\usepackage{geometry}
\usepackage{amssymb}
\usepackage{amsmath}
\usepackage{cleveref}
\usepackage{comment}
\usepackage{xcolor}
\usepackage[shortlabels]{enumitem}
\usepackage{listings}
\usepackage{blkarray}
\usepackage{lmodern}
\usepackage[english]{babel}
\usepackage{float}
\usepackage{tikz}

\newtheorem{theorem}{Theorem}[section]
\newtheorem{lemma}[theorem]{Lemma}

\theoremstyle{definition}

\newcommand{\eps}{\varepsilon}

\title{The number of arcs in $\mathbb{F}_q^2$ of a given cardinality}

\author{
Rajko Nenadov\thanks{School of Computer Science, University of Auckland, New Zealand. Email: \texttt{rajko.nenadov@auckland.ac.nz}. Research supported by the Marsden Fund of the Royal Society of New Zealand.}
}
\date{}

\begin{document}

\maketitle

\begin{abstract}
A subset of $\mathbb{F}_q^2$ is called an arc if it does not contain three collinear points. We show that there are at most $\binom{(1 + o(1))q}{m}$ arcs of size $m \gg q^{1/2} (\log q)^{3/2}$, nearly matching a trivial lower bound $\binom{q}{m}$. This was previously known to hold for $m \gg q^{2/3} (\log q)^3$, due to Bhowmick and Roche-Newton. The lower bound on $m$ is best possible up to a logarithmic factor. 
\end{abstract}

\section{Introduction}

A set of points in $\mathbb{F}_q^2$ such that no three are collinear is called an \emph{arc}. A set $\{(x, x^2) \colon x \in \mathbb{F}_q\}$ forms an arc of size $q$ and, as a subset of an arc is also an arc, we have a trivial lower bound of $2^q$ on the number of arcs in $\mathbb{F}_q^2$. More generally, for each $m \le q$ there are at least $\binom{q}{m}$ arcs of size $m$. Bhowmick and Roche-Newton \cite{bhowmick22counting} showed that there are at most $2^{(1 + o(1))q}$ arcs, thus nearly matching the trivial lower bound. For $m \gg q^{2/3} (\log q)^3$, they proved that there are at most
$$
    \binom{(1 + o(1)) q}{m}
$$
arcs of size $m$. Chen, Liu, Nie, and Zeng \cite{chen2023random} extended this by showing that for $m \gg q^{1/2} (\log q)^2$, there are at most
$$
    \binom{Kq}{m}
$$
$m$-arcs, for some constant $K > 9$. We improve this by matching the bound of Bhowmick and Roche-Newton.

\begin{theorem} \label{thm:main}
    For every $\eps > 0$ there exists $C > 0$, such that the number of arcs in $\mathbb{F}_q^2$ of size $m \ge C q^{1/2} (\log q)^{3/2}$ is at most
    $$
        \binom{(1 + \eps)q}{m}.
    $$
\end{theorem}

Using the probabilistic argument, Roche-Newton and Warren \cite{roche22arcs} showed that the number of arcs of size $m < q^{1/2}$ is at least
$$
    \binom{q^2}{m} e^{-C m^3 / q},
$$
for some absolute constant $C > 0$, which is significantly larger than $\binom{(1 + \eps)q}{m}$ for $m \ll q^{1/2}$ and $q$ sufficiently large. Therefore, the lower bound on $m$ in Theorem \ref{thm:main} is best possible up to a $O(\log^{3/2} q)$ factor.

All previous results \cite{bhowmick22counting,chen2023random,roche22arcs} on counting arcs are based on encoding triples of collinear points as hyperedges in a 3-uniform hypergraph over the vertex set $\mathbb{F}_q^2$, and then applying the machinery of hypergraph containers \cite{balogh15container,saxton15containers} to count the number of independent sets of certain size. The main difficulty in obtaining a bound of order $\binom{(1+\eps)q}{m}$ using this approach is that one needs a so-called \emph{balanced supersaturation} for induced subgraphs of size at least $(1 + \eps)q$. Briefly, one needs to show that in every such induced subgraph, there are sufficiently many hyperedges which are `nicely' distributed. We circumvent this by using Lemma \ref{lemma:graph_container} that counts independent sets in graphs which satisfy certain local density condition. The way we encode certain collinear triples using graphs, rather than hypergraphs, stems from the ideas used by Kohayakawa, Lee, R\"odl, and Samotij \cite{kohayakawa15sidon} to count Sidon sets of certain size (and their idea can be further traced to the seminal work of Kleitman and Winston \cite{kleitman82cycles}).

Related counting results about points in a general position in higher dimensions were studied by Mattheus and Van de Voorde \cite{mattheus2024upper}, also using hypergraph containers. In particular, they obtained a bound analogous to Theorem \ref{thm:main} for point sets in general position in $PG(3,q)$. The aforementioned issue of obtaining a balanced-supersaturation result does not apply here as the bounds on co-degrees coming from the properties of finite geometries suffice for their applications. The ideas presented in the present paper also apply to some of the problems considered in \cite{mattheus2024upper}, and other than perhaps improving some of their bounds by a logarithmic factor, simplify their argument and calculations.

\section{Preliminaries}

We use the following supersaturation result by Bhowmick and Roche-Newton \cite[Claim 2.6]{bhowmick22counting}.

\begin{lemma} \label{lemma:triples}
    Let $P \subseteq \mathbb{F}_q^2$ be a set of size $|P| \ge k(q+1) + x + 1$, for some $k \in \mathbb{N}$ and an integer $x$ satisfying $0 \le x < q+1$. Then each point in $P$ is contained in at least
    $$
        \binom{k}{2} (q+1) + k x
    $$
    collinear triples with all three points in $P$.
\end{lemma}

As remarked earlier, instead of hypergraph containers we use the following simpler result on the number of independent sets in locally dense graphs by Kohayakawa, Lee, R\"odl, and Samotij \cite{kohayakawa15sidon}.

\begin{lemma} \label{lemma:graph_container}
    Let $G$ be a graph on $N$ vertices, let $f$ be an integer, $0 \le \beta \le 1$, and $R$ be an integer such that
    $$
        R \ge e^{-\beta f} N.
    $$
    Suppose that for every $U \subseteq V(G)$ of size $|U| \ge R$, the number of edges in the induced subgraph $G[U]$ is at least
    $$
        e(U) \ge \beta \binom{|U|}{2}.
    $$
    Then, for any $r \ge 0$, the number of independent sets in $G$ of cardinality $f + r$ is at most
    $$
        \binom{N}{f} \binom{R}{r}.
    $$
\end{lemma}

\section{Proof of Theorem \ref{thm:main}}

Using Lemma \ref{lemma:triples}, we show that graphs encoding certain collinear triples  enjoy the local density property stated in Lemma \ref{lemma:graph_container}.

\begin{lemma} \label{lemma:graph}
    Let $F, P \subseteq \mathbb{F}_q^2$ be disjoint subsets such that $F$ is an arc and $|P| \ge (1 + \eps)q$, for some $\eps > 0$ and $q$ sufficiently large. Form the graph $G$ on the vertex set $P$ where $x, y \in P$ are connected by an edge if there is some $z \in F$ such that $(x,y,z)$ are collinear. Then
    $$
        e(G) \ge \frac{\eps |F|}{8 q} \binom{|P|}{2}.
    $$
\end{lemma}
\begin{proof}
    Set $t = \lfloor |P| / (q+1) \rfloor$. Choose a vertex $v \in F$ and set $P' = P \cup \{v\}$. By Lemma \ref{lemma:triples}, if $t \ge 2$ then the point $v$ is contained in at least
    $$
        \binom{t}{2} (q+1) \ge \frac{|P|}{8q}
    $$
    collinear triples in $P'$. If $t = 1$ then $|P| < 2(q+1)$ thus it is contained in at least
    $$
        |P| - q \ge \eps q \ge \frac{\eps |P|^2}{8q}
    $$
    collinear triples, with room to spare. In either case, $v$ contributes at least $\eps |P|^2 / (8q)$ edges in $G$. As $F$ does not contain collinear triples, each edge in $G$ is associated to  at most two points in $F$. This implies
    $$
        e(G) \ge \frac{1}{2} \cdot |F| \cdot \frac{\eps |P|^2}{8 q} \ge \frac{\eps |F|}{8 q} \binom{|P|}{2}.
    $$
\end{proof}

The proof of Theorem \ref{thm:main} now follows by combining Lemma \ref{lemma:graph} with Lemma \ref{lemma:graph_container}.

\begin{proof}[Proof of Theorem \ref{thm:main}]
    For each $F \subseteq \mathbb{F}_q^2$ of size $|F| = f := (8q \log q / \eps)^{1/2} $, we count the number of arcs which contain $F$. First, form the graph $G$ on $\mathbb{F}_q^2 \setminus F$ as described in Lemma \ref{lemma:graph}. Note that if $I \subseteq \mathbb{F}_q^2$, $F \subseteq I$, is an arc then $I \setminus F$ is an independent set in $G$. By Lemma \ref{lemma:graph_container} applied with $\beta = \eps f / (8q)$, $f$, and $R = (1 + \eps)q$, the number of independent sets in $G$ of size $r = m - f$ is upper bounded by
    $$
        \binom{q^2 - f}{f} \binom{(1 + \eps)q}{m - 2f}.
    $$
    Taking into account all possible choices for $F$, we get the following bound on the number of arcs of size $m \ge 2f$: 
    $$
        \binom{q^2}{f} \binom{q^2 - f}{f} \binom{(1 + \eps)q}{m - 2f} \le e^{4 f \log q} m^{2f} 
 \binom{(1 + \eps)q}{m} \le e^{6 f \log q} \binom{(1 + \eps)q}{m}.
    $$
    For $m \ge C q^{1/2} (\log q)^{3/2}$ we have $6 f \log q < \alpha m$, where $\alpha \rightarrow 0$ as $C \rightarrow \infty$, thus we further upper bound the last expression as follows:
    $$
        e^{\alpha m} \binom{(1 + \eps)q}{m} \le (1 + 2\alpha)^m \binom{(1 + \eps)q}{m} \le \binom{(1 + 2\eps) q}{m}.
    $$
    In the last inequality we assume that $\alpha$ is sufficiently small with respect to $\eps$ (e.g.\ $\alpha < \eps/4$ suffices).
\end{proof}

\bibliographystyle{abbrv}
\bibliography{references}

\end{document}